\documentclass[11pt]{amsart}

\usepackage{graphicx}
\usepackage{amssymb, amsthm, amsmath}
\usepackage{xcolor}
\newtheorem{theorem}{Theorem}

\newtheorem{lemma}[theorem]{Lemma}

\newtheorem*{theorem*}{Theorem}{\bf}{\it}
\newtheorem*{proposition*}{Proposition}{\bf}{\it}
\newtheorem*{observation*}{Observation}{\bf}{\it}
\newtheorem*{lemma*}{Lemma}{\bf}{\it}

\theoremstyle{definition}

\definecolor{lev}{rgb}{0.773,0.294,0.549}

\theoremstyle{remark}
\newtheorem{remark}[theorem]{Remark}

\newcommand{\R}{\mathbb R}

\newcommand{\C}{\mathbb C}

\def\XXint#1#2#3{{\setbox0=\hbox{$#1{#2#3}{\int}$ }
\vcenter{\hbox{$#2#3$ }}\kern-.6\wd0}}

\begin{document}
\title[]{ Non-isometric domains with the same 
Marvizi-Melrose invariants}

\author{Lev Buhovsky, Vadim Kaloshin}

\begin{abstract} For any strictly convex planar domain 
$\Omega \subset \R^2$ with a $C^\infty$ boundary 
one can associate an infinite sequence of spectral 
invariants introduced by Marvizi-Merlose \cite{MM}.
These invariants can generically be determined using 
the spectrum of the Dirichlet problem of the Laplace operator. 
A natural question asks if this collection is sufficient to determine 
$\Omega$ up to isometry. In this paper we give 
a counterexample, namely, we present two non-isometric 
domains $\Omega$ and $\bar \Omega$ with the same collection 
of Marvizi-Melrose invariants. Moreover, each domain 
has countably many periodic orbits $\{S^n\}_{n \geqslant 1}$ (resp. 
$\{ \bar S^n\}_{n \geqslant 1}$) of period going to infinity such that
$ S^n $ and $ \bar S^n $ have the same period and perimeter for each $ n $.
 \end{abstract}
\maketitle

Consider a  $C^\infty$ smooth  strictly convex  planar domain $\Omega \subset \mathbb R^2$. 
Let us start by introducing the {\it Length Spectrum} of a domain $\Omega$.
The length spectrum of $\Omega$ is given by the set  of 
lengths of its periodic orbits, counted with multiplicity:
\[
\mathcal  L(\Omega) := \mathbb N 
\{\text{ lengths of closed geodesics in }\Omega\} \cup 
\mathbb N \, |\partial \Omega|,
\]
where $|\partial \Omega|$ denotes the length of the boundary 
of $\Omega$. Generically this collection can be determined 
from the spectrum of the Laplace operator in $\Omega$ 
with Dirichlet boundary condition (similarly for Neumann boundary one): 
\begin{equation}\label{DirichletProblem}
\left\{
\begin{array}{l}
\Delta f = \lambda f \quad \text{in}\; \Omega \\
f|_{\partial \Omega} = 0.\end{array}\right.\\
\end{equation}
From the physical point of view, the eigenvalues $\lambda$'s
are the eigenfrequencies of the membrane $\Omega$ with
a fixed boundary. There is the following relation between 
the Laplace spectrum and the length spectrum (see e.g. 
\cite{AM,PS}). Call the function 
\[
w(t):=\sum_{\lambda_i \in spec \Delta}
\cos (t \sqrt{-\lambda_i}),
\]
the wave trace.
Then, the wave trace $w(t)$  is a well-defined 
generalized function (distribution) of $t$, smooth 
away from the length spectrum, namely, 
\begin{equation}\label{AndersonMelroseformula}
\mbox{sing. \!\!\!\! supp.} \big( w(t) \big)\subseteq  \pm \mathcal L(\Omega) \cup \{0\}.
\end{equation}
So if $l > 0$ belongs to the singular support of this distribution, then there exists 
either a closed billiard trajectory of length $l$, or a closed geodesic of length 
$l$ in the boundary of the billiard table.

Generically, equality holds in \eqref{AndersonMelroseformula}. 
More precisely, if no two distinct orbits have the same length and 
the Poincar\'e map of any periodic orbit is non-degenerate, then the singular 
support of the wave trace coincides with $\pm \mathcal L(\Omega) \cup \{0\}$ (see e.g. 
\cite{PS}). This theorem implies that,  at least for generic domains, one can 
recover the length spectrum from the Laplace one. \\

This relation between periodic orbits and spectral properties of 
the domain, immediately leads to a famous inverse spectral problem: \smallskip 

{\it Can one hear the shape of a drum?}, \smallskip 

\noindent as formulated in a very suggestive way by M. Kac \cite{Kac} (although the problem had been already stated by H. Weyl). 
More precisely, does the spectrum Spec $\Delta$ 
determine $\Omega$ up to isometry?
This question has not been completely solved yet: there 
are negative and positive answers 
(see \cite{Z-1,Z-2}).
  \smallskip 

S. Marvizi and R. Melrose \cite{MM} studied the asymptotics 
of the lengths of $n$--periodic billiard trajectories in a smooth 
strictly convex plane domain as $n \to \infty$. Let $L_n$ be 
the supremum and $l_n$ – the infimum of the perimeters of 
simple billiard $n$-gons. {The following theorem was 
proved in \cite{MM}:}

\begin{theorem*} For any positive integer $k$ we have 
\[
\lim n^k (L_n - l_n)=0 \qquad n\to \infty
\]
for any positive $k$. Moreover, $L_n$ has an asymptotic expansion as $n\to \infty$:
\[
L_n \sim∼ \ell_0 +  \sum_{k=1}^\infty \frac{\ell_k}{n^{2k}},
\]  
where $\ell_0$ is the length of the billiard table and $\ell_k$'s 
are constants, depending on the curvature of the table. 
\end{theorem*}

This collection $\{\ell_k\}_{k\ge 0}$ is sometimes called 
Marvizi-Melrose spectral invariants. 
These coefficients are closely related to expansion at the origin 
of so-called Mather's $\beta$-function (see \cite{S,T}). 
A natural question is \medskip 

{\it Do Marvizi-Melrose spectral invariants $\{\ell_k\}_{k\ge 0}$ 
determine a strictly convex domain (up to isometry)?}\medskip 

\noindent In this paper we provide a negative answer, namely,
\begin{theorem} \label{thm:main} 
There exist two $C^\infty$ strictly convex planar domain 
$\Omega,\ \Omega' \subset \mathbb R^2$ which are 
non-isometric, but have the same Marvizi-Melrose invariants. Moreover, there is a sequence $q_n\to \infty$
such that for each $n \geqslant 1$ there are periodic orbits of period 
$q_n$ for both domains of the same perimeter.
\end{theorem} \medskip 

We note that originally Marvizi-Melrose \cite{MM} 
derived their invariants $\{\ell_k\}_{k\ge 0}$ as integrated
quantities. If $s$ is the length parametrization of the boundary 
and $\rho(s)$ is its radius of curvature, then 
\[
\ell_1=-2\int_0^{\ell(\partial \Omega)}\rho^{2/3}(s)ds
\]
\[
\ell_2=\frac{1}{1080}\int_0^{\ell(\partial \Omega)}
(9\rho^{4/3}(s)+8\rho^{-8/3}\dot \rho^2(s))\,ds
\]
and so on.

We construct domains $\Omega$ and $\Omega'$
using the same ``building blocks'', namely, there is 
a partition of the boundary of both domains such that parts 
are isometric (see Figure 1 below). Since these invariants are 
integrated quantities of products of fractional powers of $\rho$ 
and 
its derivatives, the first part of our result can be derived from 
the derivations of Marvizi-Melrose \cite{MM}. We, however, 
feel that our ``moreover'' construction in Theorem \ref{thm:main} is interesting by itself.

\section{Construction of non-isometric domains with the same Marvizi-Melrose spectral invariants}
Our convex billiard tables $\Omega$ and $\Omega'$ will consist of the same ``building blocks'', 
which are ``glued'' together in different order. To be more precise, these building blocks will 
be smooth curves $ \gamma : [0,a] \rightarrow \mathbb{R}^2 $, and we will require 
that such $ \gamma $ satisfies the following:

\begin{enumerate}
\item $ \gamma $ is regular, simple and non-closed.
\item $ \gamma $ is symmetric with respect to the line passing through $ \gamma(a/2) $ and normal to 
$ \gamma $, that is, denoting by $ R_\gamma : \R^2 \rightarrow \R^2 $ the reflection of $ \R^2 $ with respect to the line passing 
through $ \gamma(a/2) $ and orthogonal to $ \gamma $ at $ \gamma(a/2) $, we have $ R_\gamma (\gamma(t)) = \gamma(a-t) $ for all $ t \in [0,a] $.

\item $ \arg \gamma'(t) $ is monotone increasing $ (\text{mod} \,\, 2\pi) $ with a positive speed, when $ t \in [0,a] $ (for given $ z \in \R^2 \setminus \{ (0,0) \} $, the notation $ \theta = \arg z $ means that $ z = (r \cos \theta, r \sin \theta) $ for some some $ r > 0 $).
\end{enumerate}

Given building blocks $ \gamma_1 : [0,a_1] \rightarrow \R^2 $ and $ \gamma_2 : [0,a_2] \rightarrow \R^2 $, we define their gluing to be a curve $ \gamma : [0,a_1 + a_2] \rightarrow \R^2 $ defined as follows. First, let $ T $ be the orientation preserving isometry of $ \R^2 $ such that for $ \tilde{\gamma}_2 := T \circ \gamma_2 $ we have $ \tilde{\gamma}_2(0) = \gamma_1(a_1) $ and $ \tilde{\gamma}_2'(0) = \gamma_1'(a_1) $. Then we define $ \gamma $ by $ \gamma(t) = \gamma_1(t) $ for $ t \in [0,a_1] $, and $ \gamma(t) = \gamma_2(t-a_1) $ for $ t \in [a_1,a_1+a_2] $. We will use the notation $ \gamma = \gamma_1 \sharp \gamma_2 $. Notice that $ \gamma $ does not have to be $ C^\infty $-smooth in general. However, in all examples that we will be considering below, this will always be the case.

%

We will call a smooth simple closed curve bounding a smooth strictly convex domain {\em a billiard table boundary}. Given a building block $ \gamma $ we can think of it as a ``wall'', with respect to which we can play billiard. We say that an angle $ \theta \in (0,\frac{\pi}{2}) $ {\em matches} $ \gamma $ if the billiard trajectory starting at $ \gamma(0) $ and making the angle $ \theta $ with $ \gamma $ at $ \gamma(0) $, {passes through 
$ \gamma(a/2) $ (and in particular, arrives to $ \gamma(a) $).}
More precisely, there exist 
$$
 t_0 = 0 < t_1 < \ldots < t_{2p} = a, 
$$ 
such that 
$$ 
\angle (\gamma(t_i) - \gamma(t_{i-1}), \gamma'(t_i)) =  \angle (\gamma'(t_i), \gamma(t_{i+1}) - \gamma(t_{i})) 
$$ 
for $ i = 1,2, \ldots, 2p-1 $, we have $ t_p = a/2 $ and 
$ \angle (\gamma'(t_0), \gamma(t_1) - \gamma(t_{0})) = \theta $. 
Clearly, in this case we also have $ \angle (\gamma(t_{2p}) - \gamma(t_{2p-1}),\gamma'(t_{2p})) = \theta $.

\begin{figure}[h!]
\centering
\includegraphics[scale=0.8]{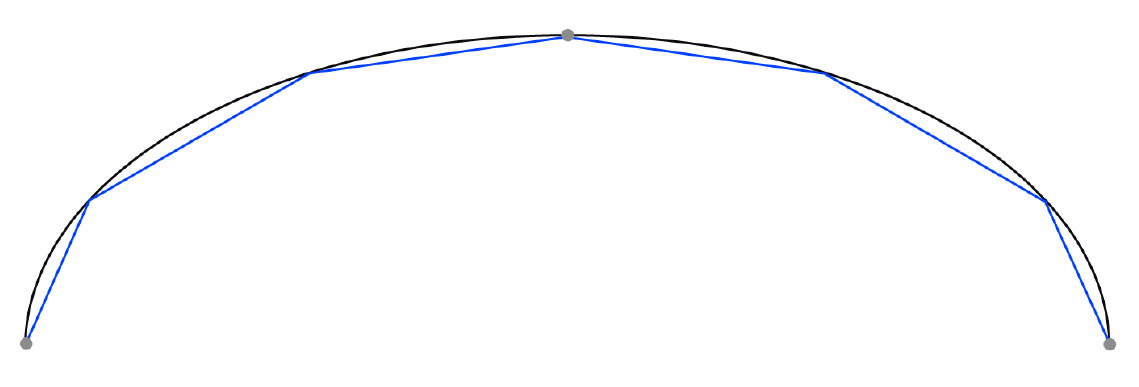}
\caption{\label{fig:rearranging-bb}} 
A billiard ball trajectory matching a building block.
\end{figure}

Assume that we have building blocks $ \gamma_1, \ldots, \gamma_n $, such that $ \gamma = \gamma_1 \sharp \cdots \sharp \gamma_n $ is a billiard table boundary. Then, for any angle $ \theta $ which matches each $ \gamma_k $, we obtain a closed billiard trajectory for $ \gamma $, which starts at $ \gamma(0) $ with angle $ \theta $. Clearly, if we have a permutation $ \gamma_{k_1}, \gamma_{k_2}, \ldots, \gamma_{k_n} $ of our building blocks, such that $ \tilde \gamma := \gamma_{k_1} \sharp \cdots \sharp \gamma_{k_n} $ is a billiard table boundary, then of course, we also obtain a closed billiard trajectory for $ \gamma $, which starts at $ \gamma(0) $ with angle $ \theta $.

The idea of our example is to find building blocks $ \gamma_1, \ldots, \gamma_n $ and a permutation $ k_1, \ldots , k_n $ of the indices $ 1, \ldots, n $, such that $ \gamma =  \gamma_1 \sharp \cdots \sharp \gamma_n $ and $ \tilde \gamma = \gamma_{k_1} \sharp \cdots \sharp \gamma_{k_n} $ are {\em different} (i.e. non-congruent) billiard table boundaries, and such that for an infinite decreasing sequence $ \theta_1, \theta_2, \ldots \in (0,\frac{\pi}{2}) $ of angles, converging to $ 0 $, each $ \theta_k $ matches each $ \gamma_j $. Then from the mentioned above, it follows that $ \gamma $ and $ \tilde \gamma $ admit an infinite sequence of pairs of simple closed billiard trajectories $ \tau_1, \tau_2, \ldots $ and $ \tilde \tau_1, \tilde \tau_2, \ldots $, such that $ \tau_k $ has the same length and the same number 
of bouncing points as $ \tilde \tau_k $, for each $ k $. { In that 
case, the Marvizi-Melrose invariants of $ \gamma $ and 
$ \tilde \gamma $ coincide, since they are defined by 
the asymptotic behaviour of lengths of simple billiard $ n $-gons.}

To construct such building blocks, we start with some $ n \geqslant 4 $ and initial collection $ \gamma_1, \ldots, \gamma_n $ of building blocks and a permutation $ k_1, \ldots , k_n $ such that $ \gamma_1 \sharp \cdots \sharp \gamma_n $ and $ \gamma_{k_1} \sharp \cdots \sharp \gamma_{k_n} $ are 
{\em non-congruent} billiard table boundaries. To obtain an example of such collection of building blocks and a permutation, one can simply look first at $ \gamma_k : [0,\frac{2\pi}{n}] \rightarrow \mathbb{R}^2 \cong \C $, 
$\gamma_k(t) = e^{it + 2\pi k i/n} $, and take any nontrivial permutation $ k_1, \ldots, k_n $ which is not of 
the form $ k_\ell = \ell + a\, (\text{mod }n)$ or $ k_\ell = a-\ell\, (\text{mod }n)$. Of course, in this case 
$\gamma_1 \sharp \cdots \sharp \gamma_n $ and 
$  \gamma_{k_1} \sharp \cdots \sharp \gamma_{k_n} $ are both congruent to the unit circle, 
but if we slightly perturb each $ \gamma_k $ on a compact subset of $ (0, \frac{2\pi}{n}) $ (keeping it to be 
a building block), then one can achieve the non-congruence of $ \gamma_1 \sharp \cdots \sharp \gamma_n $ 
and $ \gamma_{k_1} \sharp \cdots \sharp \gamma_{k_n} $.

\begin{figure}[h!]
\centering
\includegraphics[scale=0.6]{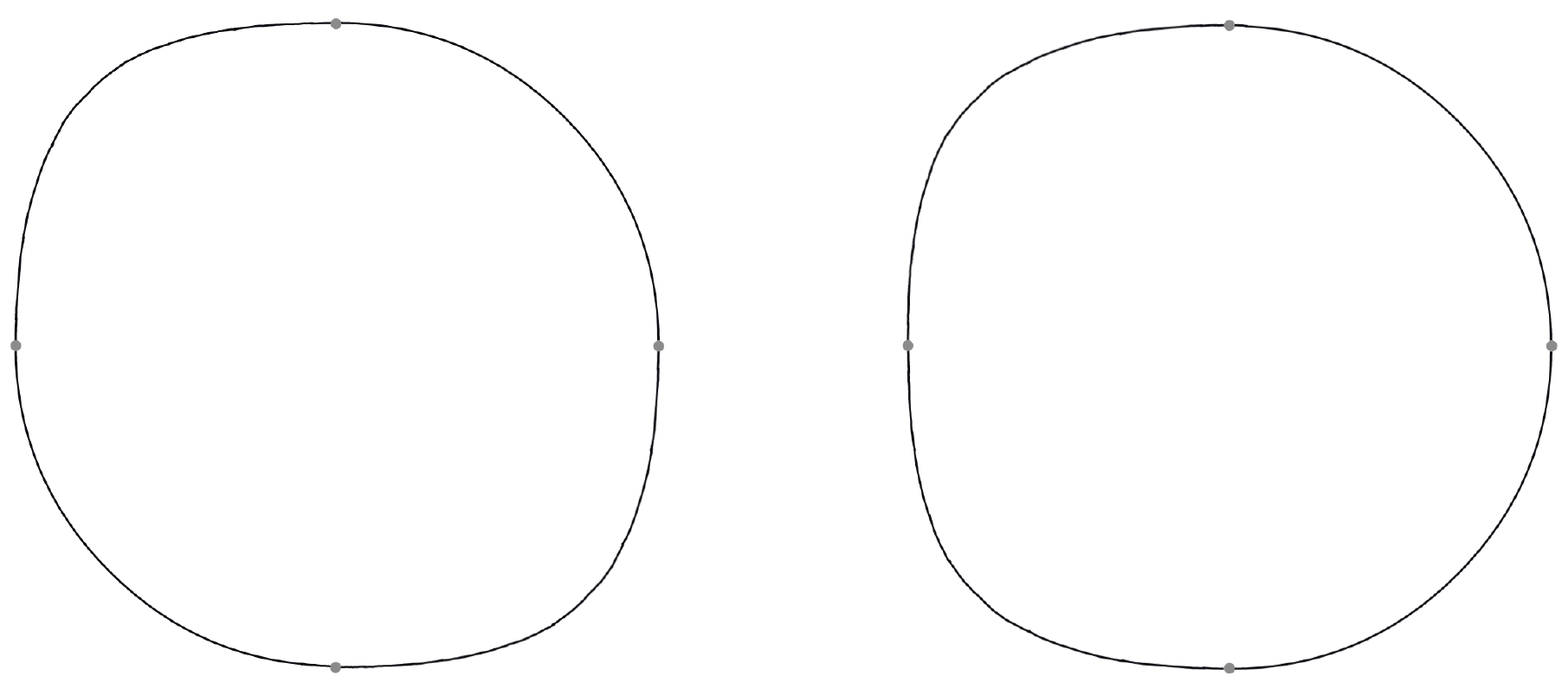}
\caption{\label{fig:rearranging-bb}} 
On the left: $ \gamma_1 \sharp \gamma_2 \sharp \gamma_3 \sharp \gamma_4 $. On the right: $ \gamma_1 \sharp \gamma_3 \sharp \gamma_2 \sharp \gamma_4 $.
\end{figure}

\begin{remark}
If in general,  $ \gamma_1, \ldots, \gamma_n $ are building blocks, $ \gamma_k : [0,a_k] \rightarrow \R^2 $, such that $ \gamma_1 \sharp \cdots \sharp \gamma_n $ is a billiard table boundary, and if we perturb each $ \gamma_k $ on a compact subset of $ (0,a_k) $ while keeping it being a building block, then $ \gamma_1 \sharp \cdots \sharp \gamma_n $ remains to be a billiard table boundary.
\end{remark}

On the next step we make infinitely many small steps on each of which we further perturb each $ \gamma_k $, such that the sizes of the perturbations decay very fast,
so that in the limit we obtain $ C^\infty $-smooth building blocks as well. We will use the following lemma:

\begin{lemma}\label{lm:match}
Let $ \gamma: [0,a] \rightarrow \R^2 $ be a building block, and let $ [b,c] \subset (0,a/2) $ be a closed interval. Then for any small enough angle $ \theta \in (0,\pi /2) $ one can find an arbitrarily $ C^\infty $-small perturbation (meaning that the size of the perturbation converges to $ 0 $ as $ \theta \rightarrow 0 $) $ \tilde \gamma $ of $ \gamma $ on $ [b,c] \cup [a-c,a-b] $, such that $ \tilde \gamma $ is a building block, and such that $ \theta $ matches $ \tilde \gamma $. 
\end{lemma}
\begin{proof} We defer a proof to Section \ref{sec:Lazutkin}.
\end{proof}

Let us describe more precisely this perturbation scheme. Fix some small enough $ \epsilon > 0 $. At the first step, by the lemma, one can find a small $ \theta_1 \in (0,\pi/2) $, such that after a small perturbation $ \tilde \gamma_k $ of each $ \gamma_k $ on a compact subset of $ (0, \frac{2\pi}{n}) $, $ \theta_1 $ matches each $ \gamma_k $. We can assume that $ \| \gamma_k - \tilde \gamma_k \|_{C^0} < \epsilon $. Now re-define each $ \gamma_k $ to be $ \tilde \gamma_k $, and pass to the second step.

Now assume that we have made $ m $ steps, and have already obtained some angles $ \theta_1,\theta_2,\ldots,\theta_m \in (0,\pi/2) $ such that each $ \theta_j $ matches each $ \gamma_k $. Let us describe the perturbation that we make on the step $ (m+1) $. For each $ 1 \leqslant k \leqslant n $, and for each $ 1 \leqslant j \leqslant m $, let $ \tau_{jk} $ be the billiard trajectory with respect to the ``wall'' $ \gamma_k $, which starts at $ \gamma_k(0) $ at the angle $ \theta_j $ with $ \gamma_k'(0) $. Now for every $ k $, look at all the bouncing points of all $ \tau_{kj} $, $ 1 \leqslant j \leqslant m $, and choose a closed interval $ [b_k,c_k] \subset (0,\pi / n) $, such that 
$ \gamma_k([b_k,c_k] \cup [2\pi / n - c_k, 2\pi / n - b_k]) $ does not contain any of these bouncing points. Then, by the lemma, for each $ k $ there exists a small perturbation $ \tilde \gamma_k $ of $ \gamma_k $ on $ [b_k,c_k] \cup [2\pi / n - c_k, 2\pi / n - b_k] $, with $ \| \gamma_k - \tilde \gamma_k \|_{C^{m}} < \epsilon / 2^m $, such that for a small angle $ \theta_{m+1} \in (0,\theta_m) $, $ \theta_{m+1} $ matches $ \gamma_k $ for every $ k $. Now re-define each $ \gamma_k $ to be $ \tilde \gamma_k $, and pass to the next step.

Note that the sizes of the perturbations decay very fast, so that our changing collection $ \gamma_1, \ldots , \gamma_n $ of building blocks, converges to some limiting collection of building blocks, which we again denote by $ \gamma_1, \ldots , \gamma_n $. Of course, by our construction procedure, now each $ \theta_j $ matches each $ \gamma_k $. Moreover, since on each step $ m $, the $ C^0 $ distance between the initial $ \gamma_k $ and the perturbed $ \gamma_k $ is smaller than $ 2\epsilon/2^m $, we conclude that the $ C^0 $ distance between each starting $ \gamma_k $ (which we had before performing the perturbation scheme) and the limiting $ \gamma_k $ is less than $ 2\epsilon $. Therefore, if $ \epsilon $ is small enough, then for the limiting building blocks $ \gamma_1, \ldots , \gamma_n $, we still get that $ \gamma_1 \sharp \cdots \sharp \gamma_n $ and $ \gamma_{k_1} \sharp \cdots \sharp \gamma_{k_n} $ are not congruent.

\subsection{A proof of Lemma \ref{lm:match}}
\label{sec:Lazutkin}

Let $\Omega$ be a strictly convex domain; recall that $s$ denotes the arc-length parametrization of 
$\partial \Omega$  and denote with $\rho(s)$ its radius of curvature at $s$. Observe that if 
$\Omega$ is $C^r$, then $\rho$ is $C^{r-2}$.
Define the \emph{Lazutkin parametrization} of the boundary:
\begin{align}
  \label{eq:Lazutkin}%
  x(s) &= C_\Omega\,\int_0^s
         \,\rho(\sigma)^{-2/3}\ d\sigma, &\text{ where }\ C_\Omega
  &=\left[\int_0^{\ell_{\partial\Omega}}\rho(\sigma)^{-2/3}d\sigma\right]^{-1}.
\end{align}
We call \emph{the Lazutkin map} the following change of
variables:
\begin{align} \label{eq:Lazutkin-map}%
  \Psi_L &:(s,\varphi)\mapsto (\,x=x(s), y(s,\varphi)=4C_\Omega
         \,\rho(s)^{1/3} \sin (\phi/2)\,).
\end{align}

Consider now the billiard map in Lazutkin coordinates $f_L = \Psi_L\circ f\circ\Psi_L^{-1}$; 
then $f_L$ has the following form (see e.g.~\cite[(1.4)]{Lazutkin}):
\begin{align}
  \label{Lazutkin-map}
  f_L : (x,y)&\to
               (x+y+y^3g(x,y),y+y^4h(x,y)),
\end{align}
where $g$ and $h$ can be expressed analytically in terms of
derivatives of the curvature radius $\rho$ up to order $3$:
hence, if $\Omega$ is $C^r$, $g,h$ are $C^{r-5}$. In the case of $C^\infty$ domains
all functions stay $C^\infty$. We need the following 

\begin{lemma}\label{lazutkin-coordinates-orbits}
  Let $\Omega$ be a $C^5$ strictly convex domain; $\gamma \subset \partial \Omega$ 
  be a connected closed segment of the boundary. For $k,N\in \mathbb Z,\ N>2$,   let 
  $(x_k,y_k) = f_L^k(x_0,y_0),$  and $\{x_k\}_{k=0}^ N \subset \gamma$.  Then there exists 
  $C>0$ depending on $\|\rho\|_{C^3}$ and independent of $N$, such that for {small enough $ y_0 $} we have
{
  \begin{align}\label{e_estimateOnLazutkin-orbits}
    \left|y_k- y_0 \right|&< \frac {C}{N^3},&
   \left|\tilde
      x_k-\tilde x_0-ky_0\right|&<\frac {C}{N^2}.
  \end{align}
} 
for any $0 < k \le N$.
\end{lemma}
 
This lemma is proven for periodic orbits in \cite{ADK}, but the same proof applies to orbits 
glancing only at a part of the boundary of a $C^5$ strictly convex domain. 

Since a building block is symmetric, 
it is sufficient to construct 
a symmetric perturbation such that an orbit emanating from $x_0=\gamma(0)$ with small angle 
$\theta >0$ will hit the symmetry point $\gamma(a/2)$.  


Consider a smooth variation $ \gamma_\delta : [0,a] \rightarrow \R^2 $ of $ \gamma $ on $ [b,c] \cup [a-c,a-b] $, when $ \delta \in (-\delta_0,\delta_0) $, such that each $ \gamma_\delta $ is a building block, $ \gamma_0 = \gamma $, and moreover the Lazutkin perimeter of $ \gamma_\delta $ is non-constant as a function of $ \delta $ on any neighbourhood of $ \delta = 0 $ (in a sense, the family $ \gamma_\delta $ varies the Lazutkin perimeter). By Lemma \ref{lazutkin-coordinates-orbits} and by the intermediate value theorem, for small enough $ y_0 $, one can always find $ \gamma_\delta $ on which we have $ x_0 = \gamma_\delta(0) $ and $ x_N = \gamma_\delta(a/2) $ for some $ N $ (for convenience, one may consider a smooth family of strictly convex domains $ \Omega_\delta $ such that $ \gamma_\delta $ is a part of $ \Omega_\delta $). Moreover, we may choose $ \delta = o(1) $ when $ y_0 \rightarrow 0 $. 

Indeed, for any given $ \lambda \in (0,\delta_0) $, choose some $ -\lambda < \delta_1 < \delta_2 < \lambda $ such that the Lazutkin perimeters of $ \gamma_{\delta_1} $ and $ \gamma_{\delta_2} $ are different. Then, for
small enough $ y_0 $ and any $ \delta \in [\delta_1,\delta_2] $ consider the billiard ball trajectory $ \tau(\delta) $ with bouncing points $ x_0(\delta) = \gamma_\delta(0), x_1(\delta), \ldots, x_k(\delta) $ on $ \gamma_\delta $, with angle $ y_0 $ at the first bouncing point $ x_0(\delta) $. Here $ k = k(\delta,y_0) $ is the last bouncing point before the trajectory escapes $ \gamma_\delta $. From Lemma \ref{lazutkin-coordinates-orbits} it follows that we have $ \lim_{y_0 \rightarrow 0} y_0 k(\delta,y_0) = 
L_\delta $, where $ L_\delta $ is the Lazutkin perimeter of $ \gamma_\delta $. Since $ L_{\delta_1} \neq L_{\delta_2} $, for small enough $ y_0 $ we get $ k(\delta_1,y_0) \neq k(\delta_2,y_0) $. Hence the function $ \delta \mapsto k(\delta,y_0) $ has a discontinuity point $ \delta_3 \in [\delta_1,\delta_2] $, at which we must have $ x_k(\delta_3) = \gamma_{\delta_3}(a/2) $, where $ k = k(\delta_3,y_0) $. This completes the proof of the lemma.

\end{document}